\documentclass[12pt]{amsart}

\input xy
\usepackage[english]{babel}
\usepackage[latin1]{inputenc}
\usepackage{graphicx}
\usepackage{amsmath}
\usepackage{amssymb}
\usepackage{amscd}
\usepackage{color}
\usepackage{oldgerm}
\usepackage{amsfonts}
\usepackage{newlfont}
\usepackage{longtable}
\usepackage{multirow}
\usepackage{xcolor}

\DeclareOldFontCommand{\rm}{\normalfont\rmfamily}{\mathrm}

\usepackage[all]{xy}

\usepackage{upref}

\def\F{\Bbb F}

\def\N{\Bbb N}

\def\ad{\operatorname{ad}}

\def\d{\operatorname{d}}

\def\dim{\operatorname{dim}}

\def\Hom{\operatorname{Hom}}

\def\Span{\operatorname{Span}}

\def\Sa{\mathcal{S}}

\def\g{\frak g}
\def\gl{\frak{gl}}
\def\h{\frak h}

\def\a{\frak{a}}

\def\j{\frak{j}}

\def\ide{\frak{i}}

\theoremstyle{plain}\swapnumbers

\newtheorem{Theorem}{Theorem}[section]
\newtheorem{Lemma}[Theorem]{Lemma}
\newtheorem{Prop}[Theorem]{Proposition}

\title{Invariant Metrics on Nilpotent Lie algebras}

\author{R. Garc\'{\i}a-Delgado}
\address{Centro de Investigaci\'on en Matem\'aticas, A.C., 
Unidad M\'erida; Carretera Sierra Papacal Chuburna Puerto Km 5, 97302, Yucat\'an, M\'exico.}
\email{rosendo.garciadelgado@alumnos.uaslp.edu.mx}


\date{\today}

\begin{document}

\maketitle

\begin{abstract}
We state criteria for a nilpotent Lie algebra $\g$ to admit an invariant metric. We use that $\g$ possesses two canonical abelian ideals $\ide(\g) \subset \mathfrak{J}(\g)$ to decompose the underlying vector space of $\g$ and then we state sufficient conditions for $\g$ to admit an invariant metric. The properties of the ideal $\mathfrak{J}(\g)$ allows to prove that if a current Lie algebra $\g \otimes \Sa$ admits an invariant metric, then there must be an invariant and non-degenerate bilinear map from $\Sa \times \Sa$ into the space of centroids of $\g/\mathfrak{J}(\g)$. We also prove that in any nilpotent Lie algebra $\g$ there exists a non-zero, symmetric and invariant bilinear form. This bilinear form allows to reconstruct $\g$ by means of an algebra with unit. We prove that this algebra is simple if and only if the bilinear form is an invariant metric on $\g$. 
\end{abstract}

\section*{Introduction}

Let $(\g,[\,\cdot\,,\,\cdot\,])$ be a Lie algebra over a field $\F$ with bracket $[\,\cdot,\,\cdot\,]$. The Lie algebra $(\g,[\,\cdot\,,\,\cdot\,])$ is said to be \emph{quadratic} if it is equipped with a non-degenerate, symmetric and bilinear form $B:\g \times \g \to \F$, which is \emph{invariant} under the adjoint representation $\ad:\g \to \gl(\g)$, that is: $B([x,y],z)=-B(y,[x,z])$, for all $x,y,z$ in $\g$. A quadratic Lie algebra is denoted by $(\g,[\,\cdot\,,\,\cdot\,],B)$ and $B$ is called an \emph{invariant metric}.
\smallskip

A semisimple Lie algebra is quadratic for the Killing form is an invariant metric defined on it. A reductive Lie algebra is an example of a non-semisimple quadratic Lie algebra; another more elaborated example is as follows: Let $(\g,[\,\cdot\,,\,\cdot\,])$ be a Lie algebra and in the vector space $\g \oplus \g^{\ast}$ define the skey-symmetric and bilinear map $[\,\cdot\,,\,\cdot\,]^{\prime}$ by $[x+\alpha,y+\beta]^{\prime}=[x,y]+\ad(x)^{\ast}(\beta)-\ad(y)^{\ast}(\alpha)$, and the symmetric bilinear form $B$ by $B(x+\alpha,y+\beta)=\alpha(y)+\beta(x)$, for all $x,y$ in $\g$ and $\alpha,\beta$ in $\g^{\ast}$. The triple $(\g \oplus \g^{\ast},[\,\cdot\,,\,\cdot\,]^{\prime},B)$ is a non-semisimple quadratic Lie algebra, as $\g^{\ast}$ is an abelian ideal. 
\smallskip

The algebraic study of quadratic non-semisimple Lie algebras can be traced on to the work of V. Kac (see \cite{Kac}), where appears an inductive way to construct solvable quadratic Lie algebras and it is also shown there that any solvable quadratic Lie algebra can be identified with this construction. In \cite{Favre}, it is stated the conditions for which these type of quadratic Lie algebras are isomorphic. In \cite{Med}, A. Medina and P. Revoy generalizes the results of V. Kac to any indecomposable, non-semisimple and quadratic Lie algebra; the proof of these results depends 
on the choice of a minimal ideal. In \cite{Kath}, I. Kath and M. Olbrich propose an alternative approach by using canonical ideals to obtain classification results. 
\smallskip

If the Killing form is non-degenerate then the Lie algebra is semisimple and therefore, is quadratic. For nilpotent Lie algebras, the Killing form is zero, so this bilinear form can not help us to determine whether a nilpotent Lie algebra is quadratic. One of the aims of this work is to provide criteria under which a nilpotent Lie algebra is quadratic. Although much has been said about nilpotent quadratic Lie algebras, we hope that these results -- which are based on classical tools and well-known results --, can take a step forward in obtaining a criterion that determines conditions for which an arbitrary Lie algebra is quadratic. 
\smallskip

If $\g$ is a nilpotent Lie algebra then there are canonical abelian ideals $\ide(\g) \subset \mathfrak{J}(\g)$ such that $[\g,\mathfrak{J}(\g)]\subset \ide(\g)$. Further, if $\g$ admits an invariant metric then $\ide(\g)^{\perp}=\mathfrak{J}(\g)$ (see \textbf{Lemma 4.2} in \cite{Kath} and \textbf{Lemma \ref{lema auxiliar 2}} below). If $\h$ is a subspace of $\g$ such that $\g=\h \oplus \mathfrak{J}(\g)$, we prove that a necessary condition for a \emph{current Lie algebra} $\g \otimes \Sa$ to admit an invariant metric --where $\Sa$ is a commutative and associtive algebra with unit--, is that there must be a bilinear map $\epsilon:\Sa \times \Sa \to \operatorname{Cent}(\h)$ with the same properties as that of an invariant metric: symmetric, invariant and non-degenerate (see \textbf{Thm. \ref{teorema current}}). In \cite{Zhu-Meng} it is proved that if $\g$ is simple then $\Sa$ admits an invariant metric. In \textbf{Cor. 1.3} of \cite{Garcia} this result is generalized to the case when $\g$ is indecomposable. The difference of the result obtained here is that the image of the bilinear map $\epsilon$ is contained in $\operatorname{Cent}(\h)$, even if $\h$ does not admit an invariant metric and $\Sa$ is not necessarily finite dimensional. In addition, in \cite{Garcia} it is assumed that the ground field $\F$ is algebraically closed.
\smallskip

In \S 1 we start by giving a review of the theory of abelian extensions of Lie algebras, as well as some standard results that we need in the sequel. Given an arbitrary Lie algebra $\h$ and a vector space $\a$, in \textbf{Prop. \ref{Prop ext abelianas 1}} we give a method to construct a Lie algebra in the space $\h \oplus \a \oplus \h^{\ast}$. In \textbf{Prop. \ref{proposicion X}} we prove that any non-abelian nilpotent quadratic Lie algebra can be constructed in this way. In \textbf{Prop. \ref{nilp cuad}} and \textbf{Prop. \ref{nilp cuad 2}} we give sufficient conditions for that such a Lie algebra to admit an invariant metric. In \S 3 we prove that if a current Lie algebra $\g \otimes \Sa$ admits an invariant metric, where $\g=\h \oplus \a \oplus \h^{\ast}$ is nilpotent and $\h$ is non-abelian, then there must be a non-degenerate bilinear map $\epsilon:\Sa \times \Sa \to \operatorname{Cent}(\h)$ such that $\epsilon(s,t)=\epsilon(st,1)$, for all $s,t$ in $\Sa$ (see \textbf{Thm. \ref{teorema current}}). In \S 4 we prove that any nilpotent Lie algebra has a non-zero, symmetric and invariant bilinear form. Using this bilinear form we construct an algebra with unit from which we can recover the original Lie algebra structure, and then we prove that this algebra is simple if and only if the bilinear form is an invariant metric (see \textbf{Thm. \ref{teorema 2}}). All the algebras and vector spaces considered in this work are finite dimensional over a unique field $\F$ of zero characteristic.

\section{abelian Extensions of Lie algebras}

Let $\rho:\g \to \gl(V)$ be a representation of a Lie algebra $(\g,[\,\cdot\,,\,\cdot\,])$ on a vector space $V$. We call $V$ a \emph{$\g$-module}. Let $C(\g;V)$ be the standard cochain complex with differential map $\d_{\rho}$ given by:
$$
\aligned
\d_{\rho}(\lambda)(x_1,\ldots,x_{p+1})&=\sum_{j=1}^{p+1}(-1)^{j-1}\rho(x_j)(\lambda(x_1,\ldots,x_{p+1}))\\
&+\sum_{i < j}(-1)^{i+j}\lambda([x_i,x_j],x_1,\ldots,x_{\hat{i}},\ldots,x_{\hat{j}},\ldots,x_{p+1}),
\endaligned
$$
where $p$ is a non-negative integer, $\lambda$ belongs to $C^p(\g;V)$ and $x_j$ belongs to $\g$, for all $1 \leq j \leq p+1$. If $\d_{\rho}(\lambda)=0$ then $\lambda$ is a $p$-cocycle in $C(\g;V)$. If the representation $\rho$ is zero we say that $V$ is a \emph{trivial $\g$-module} and we denote the differential map in $C(\h;V)$ by $\d_V$.
\smallskip

Let $\mathfrak{J}$ be an abelian ideal of $(\g,[\,\cdot\,,\,\cdot\,])$ and let $\h$ be a vector subspace complementary to $\mathfrak{J}$, that is $\g=\h \oplus \mathfrak{J}$. For each pair $x,y$ in $\h$, let $[x,y]_{\h}$ be in $\h$ and $\Lambda(x,y)$ be in $\mathfrak{J}$ such that:
\begin{equation}\label{descomposicion corchete}
[x,y]=[x,y]_{\h}+\Lambda(x,y).
\end{equation}
Let $[\,\cdot\,,\,\cdot\,]_{\h}$ be the skew-symmetric and bilinear map on $\h$ defined by $(x,y) \mapsto [x,y]_{\h}$. Since the bracket $[\,\cdot\,,\,\cdot\,]$ satisfies the Jacobi identity, then $[\,\cdot\,,\,\cdot\,]_{\h}$ is a Lie bracket in $\h$. We denote by $\Lambda:\h \times \h \to \mathfrak{J}$, the skew-symmetric and bilinear map $(x,y) \mapsto \Lambda(x,y)$. 
\smallskip

Let $R:\h \to \gl(\mathfrak{J})$ be the linear map defined by $R(x)(U)=[x,U]$, for all $x$ in $\h$ and $U$ in $\mathfrak{J}$. Let $x,y$ be in $\h$ and $U$ be in $\mathfrak{J}$; the cyclic sum on $[x,[y,U]]$ and the fact that $\mathfrak{J}$ is abelian, implies that $R$ is a representation of $(\h,[\,\cdot\,,\,\cdot\,]_{\h})$. In addition, the cyclic sum on $[x,[y,z]]$, where $x,y$ and $z$ are in $\h$, implies: 
\begin{equation*}\label{junio1}
\begin{split}
& R(x)(\Lambda(y,z)+R(y)(\Lambda(z,x))+R(z)(\Lambda(x,y))\\
&+\Lambda(x,[y,z]_{\h})+\Lambda(y,[z,x]_{\h})+\Lambda(z,[x,y]_{\h})=0.
\end{split}
\end{equation*}
Then $\Lambda$ is a $2$-cocycle in the complex $C(\h;\mathfrak{J})$.
\smallskip

Conversely, let $R:\h\to\gl(\mathfrak{J})$ be a representation of a Lie algebra $(\h,[\,\cdot\,,\,\cdot\,]_{\h})$ in a vector space $\mathfrak{J}$ and let $\Lambda$ be a $2$-cocycle in $C(\h;\mathfrak{J})$. The skew-symmetric and bilinear map $[\,\cdot\,,\,\cdot\,]$ defined in $\g=\h\oplus\mathfrak{J}$, by:
\begin{equation}\label{corchete de ext abeliana}
\aligned
\,[x,y] & = [x,y]_{\h} + \Lambda(x,y),\quad [x,U]= R(x)(U),\\
\,[U,V] & =0,\,\,\text{ for all }x,y \in \g,\,\text{ and }U,V \in \mathfrak{J},
\endaligned
\end{equation}
is a Lie bracket. The pair $(\g,[\,\cdot\,,\,\cdot\,])$ is the \emph{abelian extension of $(\h,[\,\cdot\,,\,\cdot\,]_{\h})$ by $\mathfrak{J}$ associated to
$R$ and $\Lambda$}. We denote this Lie algebra by $\h(\Lambda,R)$.
\smallskip

Suppose that $(\g,[\,\cdot\,,\,\cdot\,])$ has abelian ideals $\ide$ and $\mathfrak{J}$ such that:
\begin{equation}\label{condiciones para los ideales}
\text{(a)} \ \ \ \mathfrak{J}\ \ \text{is abelian;} \qquad\quad
\text{(b)} \ \ \  \ide \subset \mathfrak{J}; \qquad\quad
\text{(c)} \ \ \  [\g,\mathfrak{J}] \subset \ide.
\end{equation}

Let $\h$ be a subspace of $\g$ such that $\g=\h\oplus \mathfrak{J}$, and let $\a$ be a subspace of $\g$ satisfying $\mathfrak{J}=\a \oplus \ide$; thus $\g=\h\oplus\a\oplus\ide$. For $x,y$ in $\h$, we write $[x,y]=[x,y]_{\h}+\Lambda(x,y)$, where $[x,y]_{\h}$ belongs to $\h$ and $\Lambda(x,y)$ belongs to $\mathfrak{J}$ (see \eqref{descomposicion corchete}). We decompose $\Lambda$ as follows: $\Lambda(x,y)=\lambda(x,y)+\mu(x,y)$, where $\lambda(x,y)$ is in $\a$ and $\mu(x,y)$ is in $\ide$. This yields the skew-symmetric and bilinear maps $\lambda:\h \times \h \to \a$, $(x,y) \mapsto \lambda(x,y)$, and $\mu:\h \times \h \to \ide$, $(x,y) \mapsto \mu(x,y)$.
\smallskip

Due to $\mathfrak{J}$ is equal to $\a \oplus \ide$ and $[\g,\mathfrak{J}]$ is contained in $\ide$, the representation $R:\h \to \gl(\mathfrak{J})$ has the following decomposition:
$$
\varphi(x)=R(x)\vert_{\a}:\a \to \ide,\quad \text{ and }\quad \rho(x)=R(x)\vert_{\ide}:\ide \to \ide,\quad \text{ where }x \in \h.
$$
This yields linear maps $\varphi:\h \to \Hom(\a;\ide)$ and $\rho:\h \to \gl(\ide)$ such that $[x,u+\alpha]=R(x)(u+\alpha)=\varphi(x)(u)+\rho(x)(\alpha)$, for all $u$ in $\a$ and $\alpha$ in $\ide$, and we write $R=(\varphi,\rho)$. Thus, the bracket $[\,\cdot\,,\,\cdot\,]$ in $\g$ takes the form:
\begin{equation*}\label{corchete en terminos de los ideales 2}
{\aligned
\,[x,y] & = [x,y]_{\h} + \lambda(x,y) + \mu(x,y),
\\
\,[x,u+\alpha] & = \varphi(x)(u)+\rho(x)(\alpha),
\endaligned}
\end{equation*}
for all $x,y$ in $\h$, $u$ in $\a$ and $\alpha$ in $\ide$. Due to $R=(\varphi,\rho)$ is a representation then we have the identities:
\begin{align}
\label{condiciones para la desc de R-1}\varphi([x,y]_{\h}) & =\rho(x)\circ\varphi(y)-\rho(y)\circ\varphi(x),
\\
\label{condiciones para la desc de R-2}\rho([x,y]_{\h}) & =
\rho(x)\circ\rho(y)-\rho(y)\circ\rho(x),\,\,\text{ for all }x,y \in \h.
\end{align}
From \eqref{condiciones para la desc de R-2} it follows that
$\rho:\h\to\gl(\ide)$ is a representation of $(\h,[\,\cdot\,,\,\cdot\,]_{\h})$ in $\ide$. The condition \eqref{condiciones para la desc de R-1} states that $\varphi$ is a $1$-cocycle in the complex $C(\h;\Hom(\a,\ide))$ with representation $\bar{\rho}:\h\to\gl(\Hom(\a,\ide))$ defined by: 
\begin{equation}\label{definicion de rho barra}
\bar{\rho}(x)(\tau)=\rho(x)\circ \tau,\,\,\,\text{ for all }\tau \in \Hom(\a;\ide).
\end{equation}

Since $[\h,\a]$ is contained in $\ide$ and has no component in $\a$, we assume that $\a$ is a trivial $\h$-module. Due to $\mathfrak{J}=\a \oplus \ide$, we write the differential map $\d_R$ of $C(\h;\mathfrak{J})$ as follows: 

Let $x_1,\ldots,x_p$ be in $\g$ and let us write: 
$$
\textbf{X}_p=(x_1,\ldots,x_p),\,\,\text{ and }\,\,\,\textbf{X}_{p}^j=(x_1,\ldots,x_{\hat{j}}\ldots,x_p),\,\,\,\,\,1 \leq j \leq p.
$$
For a given $\Lambda$ in $C^p(\h;\mathfrak{J})$, let $\lambda(\textbf{X}_p)$ be in $\a$ and $\mu(\textbf{X}_p)$ in $\ide$ such that: $\Lambda(\textbf{X}_p)=\lambda(\textbf{X}_p)+\mu(\textbf{X}_p)$. The maps $\textbf{X}_p \mapsto \lambda(\textbf{X}_p)$ and $\textbf{X}_p \mapsto \mu(\textbf{X}_p)$, belongs to $C^p(\h;\a)$ and $C^p(\h;\ide)$, respectively; we denote these maps by $\lambda$ and $\mu$, respectively. Let $\d_{\a}$ be the differential in $C(\h;\a)$, and let $\d_{\rho}$ be the differential in $C(\h;\ide)$. Due to $R=(\varphi,\rho)$, we have:
\begin{equation}\label{descomposicion de diferencial 1}
\d_R(\Lambda)(\textbf{X}_{p+1})\!=\!\d_{\a}(\lambda)(\textbf{X}_{p+1})+\sum_{j=1}^{p+1}(-1)^{j-1}\varphi(x_j)\left(\lambda(\textbf{X}_{p+1}^j)\right)+\d_{\rho}(\mu)(\textbf{X}_{p+1})
\end{equation}
This suggests to consider the map $e_{\varphi}:C(\h;\a) \to C(\h;\ide)$, defined by
\begin{equation}\label{definicion de e varphi}
e_{\varphi}(\lambda)(\textbf{X}_{p+1})=\sum_{j=1}^{p+1}(-1)^{j-1}\varphi(x_j)\left(\lambda(\textbf{X}_{p+1}^j)\right),\,\,\text{ for all }\lambda \in C^p(\h;\a).
\end{equation}
Thus, we can rewrite \eqref{descomposicion de diferencial 1} as:
\begin{equation}\label{descomposicion de diferencial 2}
\d_R(\Lambda)(\textbf{X}_{p+1})=\d_{\a}(\lambda)(\textbf{X}_{p+1})+e_{\varphi}(\lambda)(\textbf{X}_{p+1})+\d_{\rho}(\mu)(\textbf{X}_{p+1}),
\end{equation}
where $\d_{\a}(\lambda)(\textbf{X}_{p+1})$ belongs to $\a$ and $e_{\varphi}(\lambda)(\textbf{X}_{p+1})+\d_{\rho}(\mu)(\textbf{X}_{p+1})$ belongs to $\ide$. We make the identification $\Lambda=(\lambda,\mu)$, where $\lambda$ is in $C^p(\h;\a)$ and $\mu$ is in $C^p(\h;\ide)$. By \eqref{descomposicion de diferencial 2}, the complex $C(\h;\mathfrak{J})=C(\h;\a) \oplus C(\h;\ide)$ has the differential map:
\begin{equation}\label{descomposicion de diferencial 3}
\d_R(\lambda,\mu)\!=\!(\d_{\a}(\lambda),e_{\varphi}(\lambda)+\d_{\rho}(\mu)),\,\text{ for all }(\lambda,\mu) \!\in \!C(\h;\a) \oplus C(\h;\ide).
\end{equation}
Since $\d_R^2=\d_{\a}^2=\d_{\rho}^2=0$, then $e_{\varphi}\circ \d_{\a}=-\d_{\rho}\circ e_{\varphi}$. We summarize what we obtained in the following statement.

\begin{Prop}\label{Prop ext abelianas 1}{\sl
Let $(\h,[\,\cdot\,,\,\cdot\,]_{\h})$ be a Lie algebra, $\a$ be a trivial $\h$-module, and $\rho:\h \to \gl(\ide)$ be a representation of $(\h,[\,\cdot\,,\,\cdot\,]_{\h})$ in a vector space $\ide$. Let $\varphi:\h\to\Hom(\a;\ide)$ be a $1$-cocycle in the complex $C(\h;\Hom(\a;\ide))$
associated to the representation $\bar{\rho}:\h\to\gl(\Hom(\a,\ide))$
defined by \eqref{definicion de rho barra}. Let $C(\h;\a) \oplus C(\h;\ide)$ be the complex with differential map $\d_R$ given in \eqref{descomposicion de diferencial 3}. For each $(\lambda,\mu)$ in $C^2(\h;\a) \oplus C^2(\h;\ide)$, let $[\,\cdot\,,\,\cdot\,]$ be the skew-symmetric and bilinear map in $\h \oplus \a \oplus \ide$ defined by:
$$
\aligned
&\,[x,y]=[x,y]_{\h}+\lambda(x,y)+\mu(x,y),\quad \text{ for all }x,y \in \h,\\
&\,[x,u+\alpha]=\varphi(x)(u)+\rho(x)(\alpha),\,\text{ for all }x \in \h,\,u \in \a,\text{ and }\alpha \in \ide.
\endaligned
$$
Then, $[\,\cdot\,,\,\cdot\,]$ is a Lie bracket in $\h \oplus \a \oplus \ide$ if and only if $\d_R(\lambda,\mu)=0$. 
}
\end{Prop}
\begin{proof}
It is clear that if $[\,\cdot\,,\,\cdot\,]$ is a Lie bracket then $\d_R(\lambda,\mu)=0$. Now suppose that $\d_R(\lambda,\mu)=0$, where $R=(\varphi,\rho):\h \to \gl(\a \oplus \ide)$. By hypothesis, $\varphi$ and $\rho$ satisfy \eqref{condiciones para la desc de R-1} and \eqref{condiciones para la desc de R-2}, respectively, then $R$ is a representation of $(\h,[\,\cdot\,,\,\cdot\,]_{\h})$ on $\mathfrak{J}=\a \oplus \ide$, for which $\Lambda=(\lambda,\mu)$ is a 2-cocycle in $C(\h;\a \oplus \ide)$. Then by \eqref{corchete de ext abeliana}, $[\,\cdot\,,\,\cdot\,]$ is a Lie bracket in $\h \oplus \a \oplus \ide$. 
\end{proof}

The Lie algebra in $\h \oplus \a \oplus \ide$ of \textbf{Prop. \ref{Prop ext abelianas 1}}, is an abelian extension of $(\h,[\,\cdot\,,\,\cdot\,]_{\h})$ by $\mathfrak{J}=\a \oplus \ide$, associated to the representation $R=(\varphi,\rho):\h \to \gl(\mathfrak{J})$ and the 2-cocycle $\Lambda=(\lambda,\mu)$. Thus, the Lie algebra $\h(\Lambda,R)$ of \textbf{Prop. \ref{Prop ext abelianas 1}} we denote it by $\h(\lambda,\mu,\varphi,\rho)$. In addition we have the following short exact sequence of complexes:
\begin{equation}\label{sucesion de complejos}
\xymatrix{
0  \ar[r] & C(\h;\ide)  \ar[r] & C(\h;\a) \oplus C(\h;\ide)  \ar[r] & C(\h;\a)  \ar[r]  & 0
}
\end{equation}
where $e_{\varphi}:C(\h;\a) \to C(\h;\ide)$ is the connecting homomorphism. 
We denote the cohomology group of $C(\h;\a) \oplus C(\h;\ide)$ by $H(\h;\a \oplus \ide)$. For abelian extensions, the elements in $H^2(\h,\mathfrak{J})$ are in one-to-one correspondence with isomorphism classes of extensions of $(\h,[\,\cdot\,,\,\cdot\,]_{\h})$ by $R=(\varphi,\rho)$ (see \cite{Che}, \textbf{Thm. 26.2}). In accordance with the decomposition $\mathfrak{J}=\a \oplus \ide$, we have the following.

\begin{Prop}\label{criterio cohomologico}{\sl
Let $\h(\lambda,\mu,\varphi,\rho)$ and 
$\h(\lambda^{\prime},\mu^{\prime},\varphi,\rho)$ 
be two Lie algebras constructed
as in {\bf Prop. \ref{Prop ext abelianas 1}}
associated to the same representation $R=(\varphi,\rho)$. If $(\lambda,\mu)$ and $(\lambda^\prime,\mu^\prime)$ are in the same cohomology class in $H^2(\h;\a \oplus \ide)$, then there exists an isomorphism $\Psi$ of Lie algebras between $\h(\lambda,\mu,\varphi,\rho)$ and $\h(\lambda^{\prime},\mu^{\prime},\varphi,\rho)$ such that $\Psi\vert_{\a}=\operatorname{Id}_{\a}$ and $\Psi\vert_{\ide}=\operatorname{Id}_{\ide}$.}
\end{Prop}
\begin{proof}
If $(\lambda,\mu)$ and $(\lambda^\prime,\mu^\prime)$
are in the same cohomology class in $H(\h;\a \oplus \ide)$, then there are linear maps $L:\h \to \a$ and $M:\h \to \ide$ such that $(\lambda^{\prime},\mu^{\prime})=(\lambda,\mu)+\d_R(L,M)$. Let $\Psi$ be the map between $\h(\lambda,\mu,\varphi,\rho)$ and $\h(\lambda^{\prime},\mu^{\prime},\varphi,\rho)$, defined by: $\Psi(x)=x-L(x)-M(x)$, for all $x$ in $\h$, and $\Psi(u+\alpha)=u+\alpha$, for all $u$ in $\a$ and $\alpha$ in $\ide$. Then $\Psi$ is an isomorphism of Lie algebras.
\end{proof}

We denote by $Z(\g)$ the center of $(\g,[\,\cdot\,,\,\cdot\,])$. The \emph{derived central series} $Z_1(\g) \subset Z_2(\g) \subset \cdots \subset Z_{\ell}(\g) \subset \cdots$ is defined by $Z_1(\g)=Z(\g)$ and $Z_{\ell}(\g)=\pi_{\ell-1}^{-1}\left(Z\left(\g/Z_{\ell-1}(\g)\right)\right)$, where $\ell>1$ and $\pi_{\ell-1}:\g \to \g/Z_{\ell-1}(\g)$ is the canonical projection. The \emph{descending central series} $\g^0 \supset \g^1 \supset \cdots \supset \g^{\ell} \supset \cdots$ is defined by: $\g^0=\g$ and $\g^{\ell}=[\g,\g^{\ell-1}]$, where $\ell>1$.
\smallskip

The following result defines
the ideals $\ide(\g)$ and $\mathfrak{J}(\g)$, for a nilpotent Lie algebra $(\g,[\,\cdot\,,\,\cdot\,])$. This result is a particular case of the one given in \textbf{Lemma 4.2} of \cite{Kath} and its proof can be consulted there. 

\begin{Lemma}\label{lema auxiliar 2}{\sl
Let $(\g,[\,\cdot\,,\,\cdot\,])$ be a finite dimensional nilpotent Lie algebra over a field $\F$. Let
$$
\ide(\g)=\sum_{k \in \N}Z_k(\g) \cap \g^k \ \quad\text{and}\quad\ \mathfrak{J}(\g)=\bigcap_{k \in \N}(Z_k(\g)+\g^k).
$$
Then,
\smallskip

\textbf{(i)} $\ide(\g) \subset \mathfrak{J}(\g)$. 
\smallskip

\textbf{(ii)} There exists $m\geq 1$ such that: $\displaystyle{\mathfrak{J}(\g)=Z(\g)+\sum_{k=1}^{m-1}Z_{k+1}(\g) \cap \g^k}$.

\textbf{(iii)} $\mathfrak{J}(\g)$ is abelian and $[\g,\mathfrak{J}(\g)] \subset \ide(\g)$.
\smallskip

\textbf{(iv)} If $(\g,[\,\cdot\,,\,\cdot\,])$ admits an invariant metric then $\ide(\g)^{\perp}=\mathfrak{J}(\g)$.
}
\end{Lemma}

\section{Quadratic Lie Algebras with abelian ideals $\ide$ and $\mathfrak{J}$}

For a bilinear form $B$ on a vector space $\g$, let $B^{\flat}:\g \to \g^{\ast}$ be the map defined by $B^{\flat}(x)(y)=B(x,y)$, for all $x,y$ in $\g$.
\smallskip

What follows can be applied to any quadratic Lie algebra $(\g,[\,\cdot\,,\,\cdot\,],B)$ possessing abelian ideals $\ide=\ide(\g)$ and $\mathfrak{J}=\mathfrak{J}(\g)$ such that $\ide \subset \mathfrak{J}$, $\mathfrak{J}=\ide^{\perp}$ and $[\g,\mathfrak{J}] \subset \ide$. For example, if $\g$ is nilpotent and non-abelian, then $Z(\g) \neq \{0\}$, as $\g$ has a one-dimensional ideal which is contained in $Z(\g)$. By \textbf{Lemma 3.3} of \cite{Hum}, $Z(\g) \cap [\g,\g] \neq \{0\}$, and by \textbf{Lemma \ref{lema auxiliar 2}}, $Z(\g) \subset \mathfrak{J}$, and $Z(\g) \cap [\g,\g] \subset \ide$. In addition, $\ide$ and $\mathfrak{J}$ are abelians.
\smallskip

By \textbf{Prop. \ref{Prop ext abelianas 1}} the vector space $\g$ can be written as: $\g=\h \oplus \mathfrak{J}=\h \oplus \a \oplus \ide$, where $\mathfrak{J}=\a \oplus \ide$. Due to $\ide$ is isotropic, by the Witt-decomposition we consider $\h$ as an isotropic subspace of $\g$ such that $\a^{\perp}=\h \oplus \ide$. In addition, $B$ restricted to $\a \times \a$ is non-degenerate; we denote by $B_{\a}$ the restriction $B\mid_{\a \times \a}$. 
\smallskip

From \textbf{Prop. \ref{Prop ext abelianas 1}}, there are a Lie bracket $[\,\cdot\,,\,\cdot\,]_{\h}$ in $\h$, a representation $\rho:\h \to \gl(\ide)$, a 1-cocycle $\varphi$ in $C(\h;\Hom(\a;\ide))$ associated to $\bar{\rho}:\h \to \gl(\Hom(\a;\ide))$ and a 2-cocycle $(\lambda,\mu)$ in $C(\h;\a) \oplus C(\h;\ide)$ associated to the representation $R=(\varphi,\rho):\h \to \gl(\mathfrak{J})$, such that: 
$$
[x,y]=[x,y]_{\h}+\lambda(x,y)+\mu(x,y),\quad [x,u+\alpha]=\varphi(x)(u)+\rho(x)(\alpha),
$$
for all $x,y$ in $\h$, $u$ in $\a$ and $\alpha$ in $\ide(\g)$; then $(\g,[\,\cdot\,,\,\cdot\,])=\h(\lambda,\mu,\varphi,\rho)$. In addition, due to $B$ is non-degenerate and invariant, the map
\begin{equation}\label{definicion de phi}
\phi:\ide \rightarrow  \h^{\ast},\quad \alpha  \mapsto  B^{\flat}(\alpha)\mid_{\h}
\end{equation}
is an isomorphism of $\h$-modules, that is $\ad_{\h}^{\ast}(x) \circ \phi=\phi \circ \rho(x)$, for all $x$ in $\h$. Since $\g=\h \oplus \a \oplus \ide$, the invariant metric $B$ takes the form:
\begin{equation}\label{B1}
B(x+u+\alpha,y+v+\beta)=\phi(\alpha)(y)+\phi(\beta)(x)+B_{\a}(u,v),
\end{equation}
for all $x,y$ in $\h$, $u,v$ in $\a$ and $\alpha,\beta$ in $\ide$. Using that $B$ is invariant under $[\,\cdot\,,\,\cdot\,]$ we obtain: $B(\varphi(x)(u),y)\!=\! B([x,u],y)\!=\!-B(u,[x,y])\!=\!-B_{\a}(\lambda(x,y),u)$. Then 
\begin{equation}\label{relacion varphi y lambda}
\phi(\varphi(x)(u))(y)=-B_{\a}^{\flat}\left(\lambda(x,y)\right)(u),\,\text{ for all }x,y \in \h,\,\text{ and }u \in \a.
\end{equation}

Similarly, using that $B$ is invariant under $[\,\cdot\,,\,\cdot\,]$, for $x,y,z$ in $\h$ we get:
\begin{equation}\label{mu ciclicla}
\begin{array}{rl}
\phi(\mu(x,y))(z)&=B(\mu(x,y),z)=B([x,y],z)
\\
&=B(x,[y,z])
=B(x,\mu(y,z))=\phi(\mu(y,z))(x). 
\end{array}
\end{equation}

Let $\varphi^{\prime}:\h \to \Hom(\a;\h^{\ast})$ be the map defined by $\varphi^{\prime}(x)=\phi \circ \varphi(x)$, for all $x$ in $\h$. Thus, the map $x+u+\alpha \mapsto x+u+\phi(\alpha)$, is an isomorphism between $\h(\lambda,\mu,\varphi,\rho)$ and $\h(\lambda,\phi \circ \mu,\varphi^{\prime},\ad_{\h}^{\ast})$. Further, let $B^{\prime}$ be the symmetric bilinear form on $\h \oplus \a \oplus \h^{\ast}$ defined by:
\begin{equation}\label{B2}
B^{\prime}(x+u+\alpha^{\prime},y+v+\beta^{\prime})=\alpha^{\prime}(y)+\beta^{\prime}(x)+B_{\a}(u,v),
\end{equation}
for all $x,y$ in $\h$, $u,v$ in $\a$ and $\alpha^{\prime},\beta^{\prime}$ in $\h^{\ast}$. Then $B^{\prime}$ is an invariant metric on $\h(\lambda,\phi \circ \mu,\varphi^{\prime},\ad_{\h}^{\ast})$, making it into a quadratic Lie algebra isometric to $(\g,[\,\cdot\,,\,\cdot\,],B)$ (see \eqref{B1} and \eqref{B2}). If we make the assumptions: $\varphi^{\prime}=\varphi$, $\phi \circ \mu$, and $\rho=\ad_{\h}^{\ast}$, induced by the isomorphism $\phi:\ide \to \h^{\ast}$ (see \eqref{definicion de phi}), then by \eqref{relacion varphi y lambda} and \eqref{mu ciclicla}, $\varphi$ and $\mu$ satisfy:
\begin{align}
\label{relacion varphi y lambda dual}\varphi:\h \to \Hom(\a;\h^{\ast}), & \quad \varphi(x)(u)(y)=-B_{\a}(\lambda(x,y),u),\,\text{ and }\\
\label{mu ciclica dual} \mu:\h \times \h \to \h^{\ast}, & \quad \mu(x,y)(z)=\mu(y,z)(x),
\end{align}
for all $x,y,z$ in $\h$ and $u$ in $\a$.

\begin{Prop}\label{proposicion X}{\sl
Let $\h(\lambda,\mu,\varphi,\ad_{\h}^{\ast})$ be a Lie algebra in the sense of \textbf{Prop. \ref{Prop ext abelianas 1}}, in the vector space $\h \oplus \a \oplus \h^{\ast}$. Let $B_{\a}$ be a symmetric and non-degenerate bilinear form on $\a$. The symmetric and non-degenerate bilinear form $B$ on $\h \oplus \a \oplus \h^{\ast}$ defined by:
\begin{equation}\label{metrica X}
B(x+u+\alpha,y+v+\beta)=\alpha(x)+\beta(x)+B_{\a}(u,v)
\end{equation}
is invariant if and only if the conditions in \eqref{relacion varphi y lambda dual} and \eqref{mu ciclica dual} holds. In addition, any nilpotent non-abelian quadratic Lie algebra can be identified with this construction.
}
\end{Prop}

If one are looking for conditions on $\h(\lambda,\mu,\varphi,\rho)$ to admit an invariant metric, then by the isomorphism $\phi:\ide\to \h^{\ast}$ in \eqref{definicion de phi}, we may assume that $\ide$ is equal to $\h^{\ast}$ and $\rho$ is equal to $\ad_{\h}^{\ast}$.  The condition \eqref{mu ciclica dual} imposes a restriction on $\mu$ that is easier to obtain than the one in \eqref{relacion varphi y lambda dual}. We fix a basis $\{x_1,\ldots,x_r\}$ of $\h$ and we consider its dual basis $\{\alpha_1,\ldots,\alpha_r\} \subset \h^{\ast}$, that is $\alpha_i(x_j)=\delta_{ij}$. Let $\eta_{ijk}$ be in $\F$ such that:
$$
\eta_{ijk}=\eta_{jki}=-\eta_{jik},\quad \text{ for all }1 \leq i,j,,k \leq r.
$$
Let $\mu:\h \times \h \to \h^{\ast}$ be the bilinear map defined by $\mu(x_j,x_k)=\sum_{i=1}^r\eta_{ijk}\alpha_i$, then $\mu$ is skew-symmetric and $\mu(x_j,x_k)(x_i)=\eta_{ijk}=\eta_{jki}=\mu(x_k,x_i)(x_j)$, which proves that $\mu$ satisfies \eqref{mu ciclica dual}. Therefore, we restrict to Lie algebras of the type $\h(\lambda,\mu,\varphi,\ad_{\h}^{\ast})$ where $\mu(x,y)(z)=\mu(y,z)(x)$ for all $x,y,z$ in $\h$. 
\smallskip

Let $B_{\a}:\a \times \a \to \F$ be a symmetric, non-degenerate and bilinear form on $\a$. The condition in \eqref{relacion varphi y lambda dual} is necessary if one wants to find an invariant metric on $\h(\lambda,\mu,\varphi,\ad_{\h}^{\ast})$. In the next result, we will show that we can associate to $\varphi$, a bilinear map $\lambda_{\varphi}:\h \times \h \to \a$, such that $\varphi(x)(u)(y)=-B_{\a}(\lambda_{\varphi}(x,y),u)$, regardless of whether $\h(\lambda,\mu,\varphi,\ad_{\h}^{\ast})$ admits an invariant metric.

\begin{Prop}\label{MR}{\sl
Let $\h(\lambda,\mu,\varphi,\ad_{\h}^{\ast})$ be a Lie algebra, in the sense of \textbf{Prop. \ref{Prop ext abelianas 1}}. Let $B_{\a}$ be a non-degenerate, symmetric and bilinear form on $\a$. Then, there exists a bilinear map $\lambda_{\varphi}:\h \times \h \to \a$, such that $\varphi(x)(u)(y)=-B_{\a}(\lambda_{\varphi}(x,y),u)$, for all $x,y$ in $\h$ and $u$ in $\a$.} 
\end{Prop}
\begin{proof}
Since $\varphi$ is a linear map between $\h$ and $\Hom(\a;\h^{\ast})$, for each pair $x,y$ in $\h$, let us consider the map $T(x,y):\a \to \F$, defined by:
$$
T(x,y)(u)=-\varphi(x)(u)(y),\quad \text{ for all }u \in \a.
$$
As $\varphi:\h \to \Hom(\a;\h^{\ast})$ is linear, then $T(x,y)$ is linear and belongs to $\a^{\ast}$. The map $B_{\a}^{\flat}:\a \to \a^{\ast}$ is bijective, for $B_{\a}$ is non-degenerate. Then there exists a unique element $\lambda_{\varphi}(x,y)$ in $\a$, such that $T(x,y)=B_{\a}^{\flat}(\lambda_{\varphi}(x,y))$, that is: $\varphi(x)(u)(y)=-T(x,y)(u)=-B_{\a}(\lambda_{\varphi}(x,y),u)$, for all $u$ in $\a$. Let $\lambda_{\varphi}:\h \times \h \to \a$ be the map defined by $(x,y) \mapsto \lambda_{\varphi}(x,y)$, which is bilinear because $\varphi$ is linear and $B_{\a}$ is non-degenerate. 
\end{proof}

If $\lambda_{\varphi}=\lambda$ and $\mu$ satisfies \eqref{mu ciclica dual}, by \textbf{Prop. \ref{proposicion X}}, $\h(\lambda,\mu,\varphi,\ad_{\h}^{\ast})$ admits an invariant metric. If $\lambda_{\varphi} \neq \lambda$ but $\lambda_{\varphi}-\lambda$ is a coboundary,
we can use {\bf Prop. \ref{criterio cohomologico}} to obtain 
the following criteria that determine sufficient conditions on a Lie algebra $\h(\lambda,\mu,\varphi,\ad_{\h}^{\ast}),$ to admit an invariant metric.

\begin{Prop}\label{nilp cuad}{\sl
Let $\h(\lambda,\mu,\varphi,\ad^{\ast}_{\h})$ be a Lie algebra in the sense of \textbf{Prop. \ref{Prop ext abelianas 1}} such that $\mu(x,y)(z)=\mu(y,z)(x)$ for all $x,y,z$ in $\h$. Let $\lambda_{\varphi}:\h \times \h \to \a$ be the map of \textbf{Prop. \ref{MR}} and $e_{\varphi}:C(\h;\a) \to C(\h;\h^{\ast})$ be the map of \eqref{definicion de e varphi}. If there exists a linear map $L\!:\!\h \!\to \!\a$ such that: 
$$
\textbf{(i)}\,\,\, e_{\varphi}(L)=0,\quad \text{ and }\quad \textbf{(ii)}\,\,\, \lambda_{\varphi}=\lambda+\d_{\a}(L).
$$
Then $\h(\lambda,\mu,\varphi,\ad_{\h}^{\ast})$ admits an invariant metric.}
\end{Prop}
\begin{proof}
Since $\d_R(\lambda,\mu)=0$, $e_{\varphi} \circ \d_{\a}=-\d_{\rho}\circ e_{\varphi}$ and $e_{\varphi}(L)=0$, then $\d_R(\lambda+\d_{\a}(L),\mu)=\d_R(\lambda_{\varphi},\mu)=0$, and by \textbf{Prop. \ref{Prop ext abelianas 1}}, $\h(\lambda_{\varphi},\mu,\varphi,\ad_{\h}^{\ast})$ is a Lie algebra in $\h \oplus \a \oplus \h^{\ast}$. Due to $(\lambda+\d(L),\mu)=(\lambda,\mu)+\d_R(L,0)$, by {\bf Prop. \ref{criterio cohomologico}} the Lie algebra $\h(\lambda,\mu,\varphi,\ad^{*}_{\h})$ is isomorphic to $\h(\lambda_{\varphi}),\mu,\varphi,\ad^{*}_{\h})$. Let $B$ be the symmetric and non-degenerate bilinear form on $\h \oplus \a \oplus \h^{*}$ defined in \eqref{metrica X}. As $\varphi(x)(u)(y)=-B_{\a}(\lambda_{\varphi}(x,y),u)$, then \eqref{relacion varphi y lambda dual}) holds true. Due to $\mu$ satisfies \eqref{mu ciclica dual}, by \textbf{Prop. \ref{proposicion X}}, the bilinear form $B$ in \eqref{metrica X} is an invariant metric on $\h(\lambda+\d_{\a}(L),\mu,\varphi,\ad^{*}_{\h})$, and therefore $\h(\lambda,\mu,\varphi,\ad^{*}_{\h})$ admits an invariant metric.
\end{proof}

\begin{Prop}\label{nilp cuad 2}{\sl
Let $\h(\lambda,\mu,\varphi,\ad^{\ast}_{\h})$ be a Lie algebra in the sense of \textbf{Prop. \ref{Prop ext abelianas 1}} such that $\mu(x,y)(z)=\mu(y,z)(x)$ for all $x,y,z$ in $\h$. Let $\lambda_{\varphi}:\h \times \h \to \a$ be the map of \textbf{Prop. \ref{MR}} and $B_{\a}$ be a symmetric and non-degenerate bilinear form on $\a$. If there exists a linear map $L:\h \to \a$ such that:
$$
\textbf{(i)}\,\,\lambda_{\varphi}=\lambda+\d_{\a}(L),\,\,\,\,\text{ and }\,\,\,\textbf{(ii)}\,\,B_{\a}({\lambda_{\varphi}}(y,z),L(x))=B_{\a}(\lambda_{\varphi}(x,y),L(z)),
$$
for all $x,y,z$ in $\h$. Then $\h(\lambda,\mu,\varphi,\ad^{\ast}_{\h})$ admits an invariant metric.}
\end{Prop}
\begin{proof}
Let $\mu^{\prime}=\mu+e_{\varphi}(L)$. We claim that $\mu^{\prime}(x,y)(z)=\mu^{\prime}(y,z)(x)$, for all $x,y,z$ in $\h$. Since $\mu$ satisfies the cyclic condition \eqref{mu ciclica dual}, then $\mu^{\prime}(x,y)(z)=\mu^{\prime}(y,z)(x)$ if and only if $e_{\varphi}(L)(x,y)(z)=e_{\varphi}(L)(y,z)(x)$. By definition of $e_{\varphi}$ (see \eqref{definicion de e varphi}), we have:
\begin{equation}\label{nq-1}
\begin{split}
& e_{\varphi}(L)(x,y)(z)=\varphi(x)(L(y))(z)-\varphi(y)(L(x))(z),\,\text{ and }\\
& e_{\varphi}(L)(y,z)(x)=\varphi(y)(L(z))(x)-\varphi(z)(L(y))(x).
\end{split}
\end{equation}
By \textbf{Prop. \ref{MR}}, the equations in \eqref{nq-1} can be written as:
\begin{equation}\label{nq-2}
\begin{split}
& e_{\varphi}(L)(x,y)(z)=-B_{\a}(\lambda_{\varphi}(x,z),L(y))+B_{\a}(\lambda_{\varphi}(y,z),L(x)),\,\text{ and }\\
& e_{\varphi}(L)(y,z)(x)=-B_{\a}(\lambda_{\varphi}(y,x),L(z))+B_{\a}(\lambda_{\varphi}(z,x),L(y)).
\end{split}
\end{equation}
Since $\lambda_{\varphi}=\lambda+\d_{\a}(L)$, then $\lambda_{\varphi}:\h \times \h \to \a$ is skew-symmetric. From \eqref{nq-2} we deduce that $e_{\varphi}(L)(x,y)(z)=e_{\varphi}(L)(y,z)(x)$ follows from $B_{\a}(\lambda_{\varphi}(y,z),L(x))=B_{\a}(\lambda_{\varphi}(x,y),L(z))$, which is the hypothesis in \textbf{(ii)}. Then $\mu^{\prime}(x,y)(z)=\mu^{\prime}(y,z)(x)$, for all $x,y,z$ in $\h$. 
Observe that: 
\begin{equation}\label{nq-0}
(\lambda_{\varphi},\mu^{\prime})=(\lambda+\d_{\a}(L),\mu+e_{\varphi}(L))=(\lambda,\mu)+\d_R(L,0),
\end{equation}
then $\d_R(\lambda_{\varphi},\mu^{\prime})=0$. Thus, the Lie algebra $\h(\lambda_{\varphi},\mu^{\prime},\varphi,\ad_{\h}^{\ast})$ satisfies:
$$
\varphi(x)(u)(y)=-B_{\a}(\lambda_{\varphi}(x,y),u),\quad \text{ and }\quad \mu^{\prime}(x,y)(z)=\mu^{\prime}(y,z)(x),
$$
The symmetric and non-degenerate bilinear form $B$ defined by \eqref{metrica X}, is an invariant metric on $\h(\lambda_{\varphi},\mu^{\prime},\varphi,\ad_{\h}^{\ast})$. By \textbf{Prop. \ref{criterio cohomologico}} and \eqref{nq-0}, $\h(\lambda_{\varphi},\mu^{\prime},\varphi,\ad_{\h}^{\ast})$ and $\h(\lambda,\mu,\varphi,\ad_{\h}^{\ast})$, are isomorphic, thus $\h(\lambda,\mu,\varphi,\ad_{\h}^{\ast})$ admits an invariant metric. 
\end{proof}

\subsection{Example}
Let $\h=\Span_{\F}\{x_1,x_2,x_3\}$ be the 3-dimensional Heisenberg Lie algebra, with $[x_1,x_2]_{\h}=x_3$. Let $\h^{\ast}=\Span_{\F}\{\alpha_1,\alpha_2,\alpha_3\}$, where $\alpha_i(x_j)=\delta_{ij}$. Let $\a=\Span_{\F}\{a_1,a_2,a_3\}$ with the non-degenerate, symmetric and bilinear form $B_{\a}(u_i,u_j)=\delta_{ij}$. %Let $\varphi:\h \to \Hom(\a;\h^{\ast})$ be the linear map defined by $\varphi(x_1)(u_1)=\alpha_2$, $\varphi(x_2)(u_1)=-\alpha_1$ and $\varphi(x_3)=0$. Then $\varphi([x,y]_{\h})=\ad_{\h}^{\ast}(x)\circ \varphi(y)-\ad_{\h}^{\ast}(y)\circ \varphi(y)$, for all $x,y$ in $\h$, which means that $\varphi$ is a 1-cocycle in the complex $C(\h;\Hom(\a;\h^{\ast}))$, associated to the representation $\overline{\ad_{\h}^{\ast}}:\h \to \gl(\Hom(\a;\h^{\ast}))$. Observe that $\d_{\a}(\lambda)=0$ for any $\lambda$ in $C^2(\h;\a)$, that is:
We shall use the criterion in \textbf{Prop. \ref{nilp cuad 2}} in a Lie algebra $\h(\lambda,\mu,\varphi,\ad_{\h}^{\ast})$, where $\varphi:\h \to \Hom(\a;\h^{\ast})$ is given by:
$$
\aligned
& \varphi(x_1)(u_2)=-\varphi(x_2)(u_1)=\alpha_3,\quad \varphi(x_2)(u_3)=-\varphi(x_3)(u_2)=\alpha_1,\\
& \varphi(x_3)(u_1)=-\varphi(x_1)(u_3)=\alpha_2.
\endaligned
$$
The skew-symmetric bilinear map $\lambda:\h \times \h \to \a$ is given by:
$$
\lambda(x_1,x_2)=(1+\xi)u_3,\quad \lambda(x_2,x_3)=u_1,\quad \lambda(x_3,x_1)=u_2,
$$
where $\xi$ is in $\F$. The skew-symmetric bilinear map $\mu:\h\times \h \to \h^{\ast}$ is given by: 
$$
\mu(x_1,x_2)\!=\!-2\xi\, \alpha_3,\quad \mu(x_2,x_3)\!=\!-2\xi \,\alpha_1,\quad \mu(x_3,x_1)\!=\!-2\xi \,\alpha_2.
$$
Since $\varphi(x_j)(u_j)=0$ for all $1 \leq j \leq 3$, then 
$$
\aligned
& e_{\varphi}(\lambda)(x_1,x_2,x_3)=\\
& \varphi(x_1)(\lambda(x_2,x_3))+\varphi(x_2)(\lambda(x_3,x_1))+\varphi(x_3)(\lambda(x_1,x_2))=0.
\endaligned
$$
Since $\ad_{\h}^{\ast}(x_j)(\alpha_j)=0$ for all $j$ and $\mu$ is skew-symmetric, we have:
$$
\aligned
& \d_{\ad^{\ast}_{\h}}(\mu)(x_1,x_2,x_3)=\ad_{\h}^{\ast}(x_1)(\mu(x_2,x_3))\\
&+\ad_{\h}^{\ast}(x_2)(\mu(x_3,x_1))+\ad_{\h}^{\ast}(x_3)(\mu(x_1,x_2))+\mu(x_1,[x_2,x_3]_{\h})\\
&+\mu(x_2,[x_3,x_1]_{\h})+\mu(x_3,[x_1,x_2]_{\h})=0.
\endaligned
$$
Then $\d_R(\lambda,\mu)=(\d_{\a}(\lambda),e_{\varphi}(\lambda)+\d_{\ad_{\h}^{\ast}}(\mu))=0$, and by \textbf{Prop. \ref{Prop ext abelianas 1}}, $\h(\lambda,\mu,\varphi,\ad_{\h}^{\ast})$ is a Lie algebra. The bilinear $\lambda_{\varphi}:\h \times \h \to \a$ of \textbf{Prop. \ref{MR}} is given by:
$$
\lambda_{\varphi}(x_1,x_2)=u_3,\quad \lambda_{\varphi}(x_2,x_3)=u_1,\quad \lambda_{\varphi}(x_3,x_1)=u_2.
$$
Let $L:\h \to \a$ be the linear map defined by $L(x_j)=\xi\,u_j$ for all $j$. It is straightforward to verify that:
$$
\aligned
& B_{\a}(\lambda_{\varphi}(x_1,x_2),L(x_3))=B_{\a}(\lambda_{\varphi}(x_2,x_3),L(x_1))\\
&=B_{\a}(\lambda_{\varphi}(x_3,x_1),L(x_2))=\xi.
\endaligned
$$
In addition:
$$
\aligned
& \d_{\a}(L)(x_1,x_2)=-L([x_1,x_2]_{\h})=-L(x_3)=-\xi u_3,\quad \text{ and }\\
& \d_{\a}(L)(x_2,x_3)=-L([x_2,x_3]_{\h})=0=-L([x_3,x_1]_{\h})=\d_{\a}(L)(x_3,x_1).
\endaligned
$$
Then, $\lambda_{\varphi}=\lambda+\d_{\a}(L)$. By \textbf{Prop. \ref{nilp cuad 2}}, $\h(\lambda,\mu,\varphi,\ad_{\h}^{\ast})$ admits an invariant metric. The bracket on $\h(\lambda,\mu,\varphi,\ad_{\h}^{\ast})$ is given by:
$$
\aligned
& [x_1,x_2]=x_3+(1+\xi)u_3-2\xi \alpha_3,\quad [x_2,x_3]=u_1-2\xi \alpha_1,\\
& [x_3,x_1]\!=\!u_2-2\xi \alpha_2,\quad [x_1,u_2]\!=\!-[x_2,u_1]\!=\!\alpha_3,\quad [x_2,u_3]\!=\!-[x_3,u_2]\!=\!\alpha_1,\\
& [x_3,u_1]=-[x_1,u_3]=\alpha_2,\quad [x_1,\alpha_3]=-\alpha_2,\quad [x_2,\alpha_3]=\alpha_1.
\endaligned
$$
The invariant metric $B$ on $\h(\lambda,\mu,\varphi,\ad_{\h}^{\ast})$ is: $B(x_j,u_k)=-\xi \delta_{jk}$, and $B(u_j,u_k)=B(x_j,\alpha_k)=\delta_{jk}$, for all $j,k$. 
\smallskip

It is a straightforward to verify that:
$$
\aligned
& Z(\g)=\Span_{\F}\{\alpha_1,\alpha_2\},\quad Z_2(\g)=\h^{\ast},\quad Z_3(\g)=\a \oplus \h^{\ast},\\
& Z_4(\g)=\F x_3 \oplus Z_3(\g),\quad Z_5(\g)=\g,\quad \g^1=Z_4(\g),\\
& \g^2=\a \oplus \h^{\ast},\quad \g^3=\h^{\ast},\quad \g^4=Z(\g),\quad \g^5=\{0\}
\endaligned
$$
Thus, $\ide(\g)=\h^{\ast}$ and $\j(\g)=\a \oplus \h^{\ast}$.

\section{Current nilpotent Lie algebras}

Let $\g=\h(\lambda,\mu,\varphi,\ad_{\h}^{\ast})$ be a non-abelian nilpotent quadratic Lie algebra with invariant metric $B$. By \textbf{Lemma \ref{lema auxiliar 2}}, $\mathfrak{J}(\g)=\a \oplus \h^{\ast}$ and $\ide(\g)=\h^{\ast}$. Let $\Sa$ be a finite dimensional associative and commutative algebra with unit $1$ over $\F$. In the vector space $\g \otimes \Sa$, we consider the bracket $[X \otimes s,Y \otimes t]_{\g \otimes \Sa}=[X,Y] \otimes st$, for all $X,Y$ in $\g$ and $s,t$ in $\Sa$. Then $(\g \otimes \Sa,[\,\cdot\,,\,\cdot\,]_{\g \otimes \Sa})$ is a Lie algebra which is called the \textbf{current Lie algebra of $\g$ by $\Sa$}.
\smallskip

Since $\g=\h \oplus \mathfrak{J}(\g)$, then $\g \otimes \Sa=(\h \otimes \Sa) \oplus (\mathfrak{J}(\g)\otimes \Sa)$. Thus, by \eqref{corchete de ext abeliana}, the bracket $[\,\cdot\,,\,\cdot\,]_{\g \otimes \Sa}$ in $\g \otimes \Sa$ takes the form:
\begin{equation}\label{corchete en current}
\begin{split}
& [x \otimes s,y \otimes t]_{\g \otimes \Sa}=[x,y]_{\h} \otimes st+\Lambda(x,y) \otimes st,\\
& [x \otimes s,U \otimes t]_{\g \otimes \Sa}=R(x)(U) \otimes st.
\end{split}
\end{equation}
where $x,y$ are in $\h$, $U$ is in $\mathfrak{J}(\g)$, and $s,t$ are in $\Sa$. From now on we write $[\,\cdot\,,\,\cdot\,]$ to denote the bracket $[\,\cdot\,,\,\cdot\,]_{\g \otimes \Sa}$ in $\g \otimes \Sa$.
\smallskip

Let $\mathcal{W}=\{f:\g \times \g \to \F\mid f \text{ is bilinear and invariant }\}$. Suppose that $\g \otimes \Sa$ admits an invariant metric $\bar{B}$. For each pair $s,t$ in $\Sa$, consider the map $\mathcal{L}(s,t):\g \times \g \to \F$ defined by:
$$
\mathcal{L}(s,t)(X,Y)=\bar{B}(X \otimes s,Y \otimes t),\quad \text{ for all } X,Y \in \g.
$$
We assert that $\mathcal{L}(s,t)$ belongs to $\mathcal{W}$. Indeed, let $X,Y,Z$ be in $\g$, using that $\bar{B}$ is invariant we get:
\begin{equation}\label{new version 1}
\begin{split}
& \mathcal{L}(s,t)([X,Y],Z)=\bar{B}([X,Y] \otimes s,Z \otimes t)\\
&=\bar{B}([X \otimes s,Y \otimes 1],Z \otimes t)=\bar{B}(X \otimes s,[Y\otimes 1,Z \otimes t])\\
&=\bar{B}(X \otimes s,[Y,Z] \otimes t)=\mathcal{L}(s,t)(X,[Y,Z]).
\end{split}
\end{equation}
Then $\mathcal{L}(s,t)$ belongs to $\mathcal{W}$. Following the arguments in \eqref{new version 1}, we obtain:
\begin{equation}\label{new version 2}
\begin{split}
& \mathcal{L}(s,t)([X,Y],Z)\!=\!\bar{B}([X \otimes 1,Y \otimes s],Z \otimes t)\\
&=\!\bar{B}(X \otimes 1,[Y,Z]\otimes st)=\mathcal{L}(1,st)(X,[Y,Z])
\\&=\!\mathcal{L}(1,st)([X,Y],Z).
\end{split}
\end{equation}
Using the same arguments as in \eqref{new version 2} and that $\bar{B}$ is symmetric, we get:
\begin{equation}\label{new version 3}.
\begin{split}
& \mathcal{L}(s,t)([X,Y],Z)=\bar{B}(z \otimes t,[X \otimes s,Y \otimes 1])\\
&=\bar{B}([Z,X]\otimes st,Y \otimes s)=-\bar{B}([Z \otimes s,X \otimes t],Y \otimes 1)\\
&=\bar{B}(Z \otimes s,[X,Y]\otimes t)=\mathcal{L}(s,t)(Z,[X,Y]).
\end{split}
\end{equation}
Due to $\Sa$ is commutative, from \eqref{new version 2} and \eqref{new version 3} we deduce that $\mathcal{L}(\Sa,\Sa)\vert_{[\g,\g]\times \g}$ is symmetric and invariant, that is:
\begin{equation}\label{new version 4}
\mathcal{L}(s,t)(X,Y)=\mathcal{L}(t,s)(X,Y)=\mathcal{L}(st,1)(X,Y)=\mathcal{L}(s,t)(Y,X),
\end{equation}
for all $X$ in $[\g,\g]$ and $Y$ in $\g$. 
\smallskip

Let $\operatorname{Cent}(\g)=\{T\in \mathfrak{gl}(\g)\mid T([X,Y])=[X,T(Y)]\text{ for all }X,Y \in \g\}$. Since $B$ is a non-degenerate bilinear form on $\g$, for each pair $s,t$ in $\Sa$, there exists $\Gamma(s,t)$ in $\operatorname{Cent}(\g)$ such that:
\begin{equation}\label{new version 5}
B\left(\Gamma(s,t)(X),Y\right)=\mathcal{L}(s,t)(X,Y)=\bar{B}(X \otimes s,Y \otimes t),
\end{equation}
where $X,Y$ are in $\g$. As $\mathcal{L}(\Sa,\Sa)\vert_{[\g,\g]\times g}$ is symmetric and invariant, then
$$
\aligned
\mathcal{L}(s,t)([X,Y],Z)&=B(\Gamma(s,t)([X,Y]),Z)=\mathcal{L}(st,1)([X,Y],Z)\\
&=B(\Gamma(st,1)([X,Y],Z)),\,\,\text{ for all }X,Y,Z \in \g.
\endaligned
$$
Thus $\Gamma(s,t)([X,Y])=\Gamma(st,1)([X,Y])$ and $\Gamma(\Sa,\Sa)\vert_{[\g,\g]}$ is symmetric and invariant, that is:
\begin{equation}\label{restriccion centroides}
\Gamma(s,t)\vert_{[\g,\g]}=\Gamma(t,s)\vert_{[\g,\g])}=\Gamma(st,1)\vert_{[\g,\g]},\,\text{ for all }s,t \in \Sa.
\end{equation}
It is not difficult to verify that the canonical ideal $\mathfrak{J}(\g)=\a \oplus \h^{\ast}$, is invariant under $\Gamma(s,t)$. Thus, for every $x$ in $\h$ and $U$ in $\mathfrak{J}(\g)$, there are $\epsilon(s,t,x)$ in $\h$, $\Theta(s,t,x)$ in $\mathfrak{J}(\g)$ and $\vartheta(s,t,U)$ in $\mathfrak{J}(\g)$ such that: 
$$
\Gamma(s,t)(x)=\epsilon(s,t,x)+\Theta(s,t,x),\quad \text{ and } \quad \Gamma(s,t)(U)=\vartheta(s,t,U).
$$
Let $\epsilon:\Sa \times \Sa \to \gl(\h)$ be the map defined by $\epsilon(s,t)(x)=\epsilon(s,t,x)$. Similarly, we define $\Theta:\Sa \times \Sa \to \Hom(\h,\mathfrak{J}(\g))$ by $\Theta(s,t)(x)=\Theta(s,t,x)$ and $\vartheta:\Sa \times \Sa \to \gl(\mathfrak{J}(\g))$ by $\vartheta(s,t)(U)=\vartheta(s,t,U)$.
\smallskip

Since $\Gamma(s,t)([x,y])=[x,\Gamma(s,t)(y)]$ for all $x,y$ in $\h$, by \eqref{corchete de ext abeliana} we deduce: $\epsilon(s,t)([x,y]_{\h})\!=\![x,\epsilon(s,t)(y)]_{\h}$. That is, $\epsilon(s,t)$ belongs to $\operatorname{Cent}(\h)$.
\smallskip

From \eqref{restriccion centroides} we get: 
$$
[\Gamma(s,t)(X),Y]=\Gamma(s,t)([X,Y])=\Gamma(st,1)([X,Y])=[\Gamma(st,1)(X),Y]
$$
for all $X,Y$ in $\g$, then $\Gamma(s,t)(\g)-\Gamma(st,1)(\g)\subset Z(\g)$. This implies that $\Gamma(s,t)(x)-\Gamma(st,1)(x)$ belongs to $Z(\g)\subset \mathfrak{J}(\g)$ for all $x$ in $\h$, that is:
$$
(\epsilon(s,t)(x)-\epsilon(st,1)(x))+(\Theta(s,t)(x)-\Theta(st,1)(x)) \in Z(\g)\subset \mathfrak{J}(\g)
$$
As $\Theta(s,t)(x)-\Theta(st,1)(x)$ belongs to $\mathfrak{J}(\g)$, then $\epsilon(s,t)(x)-\epsilon(st,1)(x)$ belongs to $\h\cap \mathfrak{J}(\g)=\{0\}$. Hence, $\epsilon(s,t)=\epsilon(st,1)$.
\smallskip

As $[x,U]=R(x)(U)$ and $\Gamma(s,t)([x,U])=[\Gamma(s,t)(x),U]=[x,\Gamma(s,t)(U)]$ for all $x$ in $\h$, and $U$ in $\mathfrak{J}(\g)$, by \eqref{corchete de ext abeliana} we obtain:
\begin{equation}\label{new version 7}
\vartheta(s,t)(R(x)(U))=R(\epsilon(s,t)(x))(U)=R(x)(\vartheta(s,t)(U)).
\end{equation}
Suppose there exists $s^{\prime}$ in $\Sa$ such that $\epsilon(s^{\prime},t)=0$ for all $t$ in $\Sa$. 
By \eqref{new version 7} it follows that $\Gamma(s^{\prime},t)(R(x)(U))=\vartheta(s^{\prime},t)(R(x)(U))=0$ for all $x$ in $\h$, $U$ in $\mathfrak{J}(\g)$ and $t$ in $\Sa$. By \eqref{new version 5} this amounts to say that:
\begin{equation}\label{new version 8}
\bar{B}(R(x)(u)\otimes s^{\prime},Y \otimes t)=B(\Gamma(s^{\prime},t)(R(x)(u)),Y)=0.
\end{equation}
where $Y$ is in $\g$. Due to $\bar{B}$ is non-degenerate and $Y$ and $t$ are arbitrary, from \eqref{new version 8} we deduce that $R(x)(U)\otimes s^{\prime}=0$. If $s^{\prime}$ is non-zero then $R(x)(U)=[x,U]=0$. Since $x$ and $U$ are arbitrary, it follows $[\h,\mathfrak{J}]=\{0\}$. As $\h^{\ast}=\ide(\g)\subset \mathfrak{J}(\g)$, then $[\h,\h^{\ast}]=\ad_{\h}^{\ast}(\h)(\h^{\ast})=\{0\}$, which implies that $\h$ is abelian. Therefore, if $\h$ is non-abelian, $s^{\prime}=0$.

\begin{Theorem}\label{teorema current}{\sl
Let $\g=\h(\lambda,\mu,\varphi,\ad_{\h}^{\ast})$ be a non-abelian nilpotent quadratic Lie algebra. Let $\Sa$ be a commutative and associative algebra with unit. If the current Lie algebra $\g \otimes \Sa$ admits an invariant metric and $\h$ is non-abelian, then there exists a bilinear map $\epsilon:\Sa \times \Sa \to \operatorname{Cent}(\h)$ such that:
\smallskip

\textbf{(i)} $\epsilon(s,t)=\epsilon(s\,t,1)$, for all $s,t$ in $\Sa$.
\smallskip

\textbf{(ii)} If $\epsilon(s,t)=0$ for all $t$ in $\Sa$ then $s=0$.
}
\end{Theorem}

\section{Invariant forms on nilpotent Lie algebras} 

Let $(\g,[\,\cdot\,,\,\cdot\,])$ be a finite dimensional nilpotent Lie algebra over a field $\F$ of zero characteristic. Let $\mathcal{V}$ be the vector space generated by all the symmetric and bilinear forms $f:\g \times \g \to \F$. Let $\varrho:\g \to \gl(\mathcal{V})$ be the map defined by:
\begin{equation}\label{definicion de varrho}
\varrho(x)(f)(y,z)=-f([x,y],z)-f(y,[x,z]),\quad \text{ for all }x,y,z \in \g.
\end{equation}
Then $\varrho$ is a representation of $(\g,[\,\cdot\,,\,\cdot\,])$ in $\mathcal{V}$. Due to $(\g,[\,\cdot\,,\,\cdot\,])$ is nilpotent, by \textbf{Thm. 3.3} in \cite{Hum}, there exists a non-zero $f$ in $\mathcal{V}$ such that $\varrho(x)(f)=0$ for all $x$ in $\g$. From \eqref{definicion de varrho}, this amounts to say that $f$ is an invariant form in $(\g,[\,\cdot\,,\,\cdot\,])$. 
\smallskip

For finite dimensional nilpotent Lie algebras, the above shows that there exists an invariant, symmetric and non-zero bilinear form in any nilpotent Lie algebra. To find out whether it is non-degenerate, we can associate to such a bilinear form, a non-associative algebra with unit from which we can recover the Lie algebra structure in $\g$ and we can prove that this non-associative algebra is simple if and only if the bilinear form is non-degenerate.

\begin{Theorem}\label{teorema 2}{\sl
Let $(\g,[\,\cdot\,,\,\cdot\,])$ be a finite dimensional nilpotent Lie algebra over field $\F$ of zero characteristic. Let $f$ be an invariant, symmetric and non-zero bilinear form on $(\g,[\,\cdot\,,\,\cdot\,])$. Let $\mathcal{A}_f=\F \times \g$ be the algebra with product defined by: 
\begin{equation}\label{definicion de producto en A}
(\xi,x)(\eta,y)=(\,\xi \,\eta+f(x,y)\,,\,\xi\, y+\eta\, x+\frac{1}{2}[x,y]\,),
\end{equation}
for all $\xi,\eta$ in $\F$ and $x,y$ in $\g$. 
\smallskip

\textbf{(i)} If $\dim \g>1$, then $f$ is non-degenerate if and only if $\mathcal{A}_f$ is a simple algebra.
\smallskip

\textbf{(ii)} Let $[\,\cdot\,,\,\cdot\,]_f$ be the commutator of \eqref{definicion de producto en A}. Then $\mathcal{A}_f/\F (1,0)$ with bracket induced by $[\,\cdot\,,\,\cdot\,]_f$, is a Lie algebra isomorphic to $(\g,[\,\cdot\,,\,\cdot\,])$.
}
\end{Theorem}
\begin{proof}
$\,$

\textbf{(i)} We define equality, addition and multiplication by scalars
in $\mathcal{A}_{f}$ in the obvious manner. Suppose $f$ is non-degenerate and let $I$ be a non-zero right ideal of $\mathcal{A}_f$. We shall prove that $(1,0)$ belongs to $I$. We assert that there exists an element $(\xi,x)$ in $I$ with $\xi \neq 0$. Let $(\xi,x)$ be a non-zero element in $I$. If $\xi=0$, then $x$ is non-zero. Due to $f$ is non-degenerate, there exists $y$ in $\g$ such that $f(x,y)$ is non-zero. By \eqref{definicion de producto en A}, we have that $(0,x)(\eta,y)=\left(f(x,y),\eta x+\frac{1}{2}[x,y]\right)$ belongs to $I$, which proves our assertion.
\smallskip

Let $(\xi,x)$ be in $I$, then $(\xi,x)(-\xi,x)=(-\xi^2+f(x,x),0)$ is in $I$. If $-\xi^2+f(x,x)$ is non-zero, then $(1,0)$ belongs to $I$ and consequently $I=\mathcal{A}_f$. Thus, from now on we assume the following:
\begin{equation}\label{clubsuit}
\text{ For all }(\xi,x) \text{ in } I,\quad \xi^2=f(x,x)\tag{$\clubsuit$}.
\end{equation}
Let $(\xi,x)$ be in $I$, with $\xi \neq 0$ and $x \neq 0$ (if $x=0$ we are done), and let $(\eta,y)$ be an arbitrary element in $\mathcal{A}_f$. Then $(\xi,x)(\eta,y)=(\xi \eta+f(x,y),\xi y+\eta x+\frac{1}{2}[x,y])$ belongs to $I$. Applying \eqref{clubsuit}, we obtain:
\begin{equation}\label{T2-1}
f(x,y)^2=\xi^2 f(y,y)+\frac{1}{4}f([x,y],[x,y]).
\end{equation}
Due to $y$ is arbitrary, we apply \eqref{T2-1} to $y+z$ in $\g$ and using that $f$ is invariant we get:
\begin{equation}\label{T2-2}
4f(x,y)f(x,z)=4\xi^2f(y,z)-f([x,[x,y]],z).
\end{equation}
Using that $f$ is non-degenerate, from \eqref{T2-2} it follows:
\begin{equation}\label{T2-3}
4f(x,y)x=4\xi^2 y-\ad(x)^2(y)
\end{equation}
Applying the adjoint representation $\ad(x)$ in \eqref{T2-3}, we have:
\begin{equation}\label{T2-4}
\ad(x)^2(\ad(x)(y))=4\xi^2\ad(x)(y).
\end{equation}
If $\ad(x)^2=0$ then from \eqref{T2-2} we get: 
\begin{equation}\label{T2-5}
f(x,y)f(x,z)=\xi^2 f(y,z)
\end{equation}
As $f$ is non-degenerate and $\xi$ is non-zero, from \eqref{T2-5} we obtain $y=\xi^{-2}f(x,y)x$. Due to $y$ is arbitrary, we deduce that $\g=\F x$ and $\dim \g=1$, a contradiction.
\smallskip

If $\ad(x)^2$ is non-zero, then $\ad(x) \neq 0$ and there exists $y$ in $\g$ such that $\ad(x)(y) \neq 0$. Then $\ad(x)(y) \neq 0$ is an eigenvector of $\ad(x)$ with eigenvalue $4 \xi^2 \neq 0$ by \eqref{T2-4}, which contradicts that $\ad(x)$ is nilpotent.
\smallskip

Now suppose that $\mathcal{A}_f$ is a simple algebra and we will prove that $f$ is non-degenerate. Let $R$ be the subspace of $\g$ generated by those elements $x$ such that $f(x,y)=0$ for all $y$ in $\g$. We claim that $\{0\} \times R$ is a two-sided ideal of $\mathcal{A}_f$. Indeed, as $f$ is invariant it is clear that $R$ is an ideal of $\g$. Let $x$ be in $R$ and $(\eta,y)$ be in $\mathcal{A}_f$. Then $(0,x)(\eta,y)=(0,\eta x+\frac{1}{2}[x,y])=(\eta,y)(0,x)$  belongs to $\{0\} \times R$. Due to $\mathcal{A}_f$ is simple and $f$ is non-zero, then $\{0\} \times R$ is zero, which implies that $R=\{0\}$ and $f$ is non-degenerate. 
\smallskip

\textbf{(ii)} Let $(\xi,x)$ and $(\eta,y)$ be in $\mathcal{A}_f$, then $[(\xi,x),(\eta,y)]_f=(0,[x,y])$. The kernel of the projection map $(\xi,x) \mapsto x$, between $\mathcal{A}_f$ and $\g$, is $\F (1,0)$. Therefore, the map $(\xi,x)+\F (1,0) \mapsto x$, is a Lie algebra isomorphism between $\mathcal{A}_f/\F (1,0)$, with bracket induced by $[\,\cdot\,,\,\cdot\,]_f$, and $(\g,[\,\cdot\,,\,\cdot\,])$.

\end{proof}

\textbf{Thm. \ref{teorema 2}} can be considered as a generalization of the results given in \cite{Jenner}, from the abelian to the nilpotent case.  

\section*{Acknowledgements}

The author thanks the support provided by post-doctoral fellowship CONAHCYT 769309. The author declares that he has no conflicts to disclose.

\bibliographystyle{amsalpha}

\end{document}